\documentclass[11pt,reqno]{amsart}

\usepackage{amsmath,amssymb,mathtools,verbatim,enumitem,appendix}%,kpfonts}
\usepackage[widespace,upright]{fourier}
\usepackage[backrefs]{amsrefs}
\usepackage[protrusion=true,babel=true]{microtype}
\usepackage[english]{babel}
\usepackage[margin=1in]{geometry}
\usepackage[onehalfspacing]{setspace}
\usepackage[pdfusetitle,colorlinks,pagebackref,hypertexnames=false,bookmarks=false]{hyperref}
\numberwithin{equation}{section}							% must call it before cleveref
\usepackage[nameinlink,noabbrev]{cleveref}
\expandafter\def\csname ver@etex.sty\endcsname{3000/12/31}	% this fixes a random, irrelevant warning that clever throws

\usepackage{autonum}										% must call after cleveref

\allowdisplaybreaks

\let\originalleft\left
\let\originalright\right
\renewcommand{\left}{\mathopen{}\mathclose\bgroup\originalleft}
\renewcommand{\right}{\aftergroup\egroup\originalright}

\makeatletter
\renewcommand*{\eqref}[1]{\hyperref[{#1}]{\textup{\tagform@{\ref*{#1}}}}}
\makeatother

\newcommand{\eqdef}{\mathrel{\vcenter{\baselineskip0.5ex\lineskiplimit0pt\hbox{.}\hbox{.}}}=}

\newtheorem{theorem}{Theorem}[section]

\newtheorem*{acknowledgment}{Acknowledgment}

\newtheorem{corollary}[theorem]{Corollary}

\newtheorem{lemma}[theorem]{Lemma}

\newtheorem{remark}[theorem]{Remark}

\crefname{theorem}{Theorem}{Theorems}						% label for Theorems
\creflabelformat{theorem}{#2{#1}#3}							% label format for 'theorem'
\crefname{Mtheorem}{Main Theorem}{Main Theorems}			% label for the Main Theorems
\creflabelformat{Mtheorem}{#2{#1}#3}						% label format for 'Mtheorem'
\crefname{lemma}{Lemma}{Lemmata}							% label for Lemmata
\creflabelformat{lemma}{#2{#1}#3}							% label format for 'lemma'
\crefname{corollary}{Corollary}{Corollaries}				% label for Corollaries
\creflabelformat{corollary}{#2{#1}#3}						% label format for 'corollary'
\crefname{proposition}{Proposition}{Propositions}			% label for Propositions
\creflabelformat{proposition}{#2{#1}#3}						% label format for 'proposition'
\crefname{ineq}{inequality}{inequalities}					% label for inequalities
\creflabelformat{ineq}{#2{\upshape(#1)}#3}					% label format for 'ineq'
\crefname{cond}{condition}{conditions}						% label for conditions
\creflabelformat{cond}{#2{\upshape(#1)}#3}					% label format for 'cond'
\crefname{hypoth}{Hypothesis}{Hypotheses}					% label for Hypotheses
\creflabelformat{hypoth}{#2{#1}#3}							% label format for 'hypoth'
\crefname{definition}{Definition}{Definitions}				% label for Definitions
\creflabelformat{def}{#2{#1}#3}								% label format for 'def'
\crefname{appsec}{Appendix}{Appendices}

\def\rl{\mathbb{R}}
\def\S{\mathbb{S}}

\def\cC{\mathcal{C}}

\def\cF{\mathcal{F}}

\def\cH{\mathcal{H}}
\def\cI{\mathcal{I}}

\def\dA{\: \mathrm{dA}}
\def\Ar{\mathrm{Area}}

\def\cpt{\mathrm{cpt}}
\def\del{\partial}
\def\KS{\mathrm{KS}}
\def\loc{\mathrm{loc}}
\def\rd{\operatorname{d\!}{}}

\title{Stationary solutions to the Keller--Segel equation on curved planes}
\date{\today}

\author{\'Akos Nagy}
\address[\'Akos Nagy]{University of California, Santa Barbara}
\urladdr{\href{https://akosnagy.com}{akosnagy.com}}
\email{\href{mailto:contact@akosnagy.com}{contact@akosnagy.com}}

\calclayout
\hypersetup{
	unicode			= true,
	pdffitwindow	= true,
	pdftoolbar		= false,
	pdfmenubar		= false,
	pdfstartview	= {FitH},
	hypertexnames	= false,
	colorlinks		= true,
	linkcolor		= black,
	citecolor		= black,
	filecolor		= black,
	urlcolor		= blue
}

\begin{document}

\begin{abstract}
	We study stationary solutions to the Keller--Segel equation on curved planes.

	We prove the necessity of the mass being $8 \pi$ and a sharp decay bound. Notably, our results do not require the solutions to have a finite second moment, and thus are novel already in the flat case.

	Furthermore, we provide a correspondence between stationary solutions to the static Keller--Segel equation on curved planes and positively curved Riemannian metrics on the sphere. We use this duality to show the nonexistence of solutions in certain situations. In particular, we show the existence of metrics, arbitrarily close to the flat one on the plane, that do not support stationary solutions to the static Keller--Segel equation (with any mass).

	Finally, as a complementary result, we prove a curved version of the logarithmic Hardy--Littlewood--Sobolev inequality and use it to show that the Keller--Segel free energy is bounded from below exactly when the mass is $8 \pi$, even in the curved case.
\end{abstract}

\maketitle

\section{Introduction}

The Keller--Segel type equations describe \emph{chemotaxis}, that is the movement of organisms (typically bacteria) in the presence of a (chemical) substance. The simplest Keller--Segel system is a pair of equations on the density of the organisms, $\varrho$, and the concentration of the substance, $c$, both of which are functions on $[0, T) \times \rl^n$. Furthermore, $\varrho$ is assumed to be nonnegative and integrable. Together they satisfy the (parabolic-elliptic) Keller--Segel equations:
\begin{subequations}
\begin{align}
	\left( \del_t + \Delta \right) \varrho	&= \rd^* \left( \varrho \rd c \right), \label{eq:KS1} \\
	\Delta c								&= \varrho, \label{eq:KS2}
\end{align}
\end{subequations}
where $\rd$ is the gradient, $\rd^*$ is its $L^2$-dual (the divergence), and $\Delta = \rd^* \rd$. The mass of $\varrho$ is
\begin{equation}
	m \eqdef \int\limits_{\rl^d} \varrho (x) \rd^n x \in \rl_+,
\end{equation}
is a conserved quantity.

\smallskip

Stationary solutions to \cref{eq:KS1,eq:KS2} satisfy
\begin{subequations}
\begin{align}
	\Delta \varrho	&= \rd^* \left( \varrho \rd c \right), \label{eq:static_KS1} \\
	\Delta c		&= \varrho. \label{eq:static_KS2}
\end{align}
\end{subequations}
There is some ambiguity in the choice of $c$ in \cref{eq:static_KS1,eq:static_KS2}, and the standard choice is to use the Green's function of the Laplacian to eliminate $c$ and \cref{eq:static_KS2} via
\begin{equation}
	c_\varrho (x) \eqdef - \frac{1}{2 \pi} \int\limits_{\rl^2} \ln \left( |x - y| \right) \varrho (y) \rd^2 y,
\end{equation}
and use the single equation
\begin{equation}
	\Delta \varrho = \rd^* \left( \varrho \rd c_\varrho \right). \label{eq:static_KS}
\end{equation}
There is a well-known family of solutions to \cref{eq:static_KS}: Let $\lambda \in \rl_+$ and $x_\star \in \rl^2$ be arbitrary, and define
\begin{equation}
	\varrho_{\lambda, x_\star} \eqdef \frac{8 \lambda^2}{\left( \lambda^2 + |x - x_\star|^2 \right)^2}. \label{eq:flat_solution}
\end{equation}
Then $\varrho_{\lambda, x_\star}$ is a solution to \cref{eq:static_KS} with $m = 8 \pi$.

When the metric is the standard, euclidean metric on $\rl^2$, the literature of \cref{eq:KS1,eq:KS2} and \cref{eq:static_KS} is vast; the Reader may find good introductions in \cites{DP04,BDP06,CD12}. Very little is known about the curved case, that is, when the underlying space is not the (flat) plane. We remark here the work of \cite{MP20}, where the authors considered \cref{eq:KS1,eq:KS2} on the hyperbolic plane.

In this paper, we study the case when the metric is conformally equivalent to the flat metric and the conformal factor has the form $e^{2 \varphi}$, where $\varphi$ is smooth and compactly supported. Let us note that some of our results are novel already in the flat $(\varphi = 0)$ case. In particular, we prove that (under very mild hypotheses), solutions to \cref{eq:static_KS} have mass $8 \pi$.

\subsection*{Outline of the paper}

In \Cref{sec:static_KS}, we introduce the static Keller--Segel equation on the curved plane $\left( \rl^2, e^{2 \varphi} g_0 \right)$. In \Cref{sec:reduction}, we prove in \Cref{theorem:reduced_KS_static} that, under mild hypothesis of the growth of $\varrho$, the static Keller--Segel equation can be reduced to a simpler equation (see in \cref{eq:reduced_KS_static}). Furthermore, in \Cref{cor:rho_bounds}, we give sharp bounds on the decay rate of $\varrho$ and in \Cref{theorem:critical_mass} we show that a (nonzero) solution must have $m = 8 \pi$. In \Cref{sec:KW}, we explore a connection between solutions to the (reduced) static Keller--Segel equation and Kazdan--Warner equation on the round sphere. As an application, we prove the nonexistence of solutions for certain conformal factors in \Cref{theorem:nonexistence}. Finally, in \Cref{sec:curved_log-HLS}, we prove the logarithmic Hardy--Littlewood--Sobolev for $\left( \rl^2, e^{2 \varphi} g_0 \right)$ and in \Cref{sec:KS_fe}, as an application, we show that, as in the flat case, the Keller--Segel free energy on $\left( \rl^2, e^{2 \varphi} g_0 \right)$ is bounded from below only when $m = 8 \pi$.

\medskip

\begin{acknowledgment}
	I thank Michael Sigal for introducing me to the topic and for his initial guidance. I also thank the referee for their helpful recommendations.
\end{acknowledgment}

\bigskip

\section{The curved, static Keller--Segel equation}
\label{sec:static_KS}

Let $g_0$ be the standard metric on $\rl^2$, let $\varphi \in C_\cpt^\infty \left( \rl^2 \right)$, let $g_\varphi \eqdef e^{2 \varphi} g_0$. Let $L_k^p \left( \rl^2, g_\varphi \right)$ be Banach space of functions on $\rl^2$ that are $L_k^p$ with respect to $g_\varphi$. Note that the properties of being bounded in $L_{1, \loc}^2$ are independent of the chosen metric. Finally, let $L_+^1  \left( \rl^2, g_\varphi \right) \subseteq L^1 \left( \rl^2, g_\varphi \right)$ be the space of almost everywhere positive functions.

The area form and the Laplacian behave under a conformal change via
\begin{equation}
	\dA_\varphi = e^{2 \varphi} \dA_0 \quad \& \quad \Delta_\varphi = e^{- 2 \varphi} \Delta_0.
\end{equation}
Thus the Green's function is conformally invariant:
\begin{equation}
	G (x, y) = - \frac{1}{2 \pi} \ln \left( |x - y| \right).
\end{equation}
For any $\varrho \in L_+^1 \left( \rl^2, g_\varphi \right)$, let
\begin{equation}
	c_{\varphi, \varrho} \eqdef \int\limits_{\rl^2} G (\cdot, y) \varrho (y) \dA_\varphi (y), \label{eq:c_def}
\end{equation}
when the integral exists. Assume that the function $\varrho \in L_+^1 \left( \rl^2, g_\varphi \right) \cap L_{1, \loc}^2$ is such that $c_{\varphi, \varrho}$ is defined on $\rl^2$. Then $\varrho$ is a solution to the \emph{static Keller--Segel equation} on $\left( \rl^2, g_\varphi \right)$ if it solves (the weak version of)
\begin{equation}
	\Delta_\varphi \varrho - \rd^* \left( \varrho \rd c_{\varphi, \varrho} \right) = 0. \label{eq:KS_static}
\end{equation}
In the next section we prove that, under mild hypotheses, \cref{eq:KS_static} is equivalent to the simpler
\begin{equation}
	\rd \left( \ln \left( \varrho \right) - c_{\varphi, \varrho} \right) = 0. \label{eq:reduced_KS_static}
\end{equation}
We call \cref{eq:reduced_KS_static} the \emph{reduced, static Keller--Segel equation}.

\smallskip

In applications it is always assumed that $\varrho$ has finite mass. Furthermore, the minimal regularity needed for the weak version of \cref{eq:KS_static} is $L_{1, \loc}^2$ and the fact that $c_{\varphi, \varrho}$ is defined. Finally, we impose the finiteness of the entropy: $\varrho \ln \left( \varrho \right) \in L^1 \left( \rl^2, g_\varphi \right)$. This is implied by, for example, the finiteness of the Keller--Segel free energy; cf \Cref{sec:KS_fe}. With that in mind, we define the \emph{(curved) Keller--Segel configuration space} as:
\begin{equation}
	\cC_\KS (m, \varphi) \eqdef \left\{ \ \varrho \in L_+^1 \left( \rl^2, g_\varphi \right) \cap L_{1, \loc}^2 \ \middle| \ \begin{array}{l} \varrho \ln \left( \varrho \right) \in L^1 \left( \rl^2, g_\varphi \right), \\ \| \varrho \|_{L^1 \left( \rl^2, g_\varphi \right)} = m, \\ c_{\varphi, \varrho} \mbox{ is defined everywhere.} \end{array} \right\}. \label{eq:KS_config_space}
\end{equation}
Let $r (x) \eqdef |x|$ be the euclidean radial function. First we prove a bound on $c_{\varphi, \varrho}$.

\smallskip

\begin{lemma}
	\label{lemma:c_bound}
	Let $\varrho \in \cC_\KS (m, \varphi)$ be a solution of the static Keller--Segel \cref{eq:KS_static}. Then the function $c_{\varphi, \varrho} + \tfrac{m}{4 \pi} \ln \left( 1 + r^2 \right)$ is bounded.
\end{lemma}

\begin{proof}
	As $\Delta_\varphi c_{\varphi, \varrho} \in L^1 \left( B_1 (0), g_\varphi \right)$, it is enough to prove, without any loss of generality, the boundedness of $c_{\varphi, \varrho} + \tfrac{m}{2 \pi} \ln \left( r \right)$, when $r \geqslant 1$.

	Since $c_{\varphi, \varrho} (0) = - \tfrac{1}{2 \pi} \int_{\rl^2} \varrho \ln (r) \dA_\varphi$ is finite, we have that
	\begin{align}
		c_{\varphi, \varrho} (x)	&\leqslant o (1) - \frac{1}{2 \pi} \int\limits_{B_{|x|/2} (x)} \ln \left( |x - y| \right) \varrho (y) \dA_\varphi (y) \\
									&\leqslant o (1) - \frac{1}{2 \pi} \ln \left( |x| \right) \int\limits_{B_{|x|/2} (x)} \varrho \dA_\varphi \\
									&\leqslant O (1) - \frac{m}{2 \pi} \ln \left( |x| \right) + \frac{1}{2 \pi} \int\limits_{\rl^2 - B_{|x|/2} (x)} \ln (r) \varrho \dA_\varphi \\
									&\leqslant O (1) - \frac{m}{2 \pi} \ln \left( |x| \right).
	\end{align}
	This proves the upper bound.

	In order to get the lower bound, let us use Jensen's inequality to get
	\begin{equation}
		c_{\varphi, \varrho} (x) - c_{\varphi, \varrho} (0) = - \frac{m}{2 \pi} \int\limits_{\rl^2} \ln \left( \frac{|x - y|}{|y|} \right) \frac{\varrho (y) \dA_\varphi (y)}{m} \geqslant - \frac{m}{2 \pi} \ln \left( \ \int\limits_{\rl^2} \frac{|x - y|}{|y|} \varrho (y) \dA_\varphi (y) \right) + \frac{m}{2 \pi} \ln (m).
	\end{equation}
	Since $\varrho \in L_{1, \loc}^2$, we get that there exists $\delta > 0$, such that for all $p > 1$, $\varrho \in L^p \left( B_\delta (0) \right)$. We can assume that $\delta \leqslant 1$. Since for all $q \in [1, 2)$, $r^{- 1} \in L^q \left( B_\delta (0) \right)$ and $\tfrac{|x - y|}{|y|} \leqslant \tfrac{\sqrt{|x|^2 + \delta^2}}{\delta}$ on $\rl^2 - B_\delta (0)$, we get that, for any $p > 1$, that
	\begin{align}
		\int\limits_{\rl^2} \frac{|x - y|}{|y|} \varrho (y) \dA_\varphi (y)	&= \left( \ \int\limits_{B_\delta (0)} + \int\limits_{\rl^2 - B_\delta (0)} \right) \frac{|x - y|}{|y|} \varrho (y) \dA_\varphi (y) \\
		&\leqslant \left( e^{2 \| \varphi \|_{L^\infty \left( B_\delta (0) \right)}} \| \varrho \|_{L^p \left( B_\delta (0) \right)} \| r^{- 1} \|_{L^{\frac{p}{p - 1}} \left( B_\delta (0) \right)} + m \right) \frac{\sqrt{|x|^2 + \delta^2}}{\delta}.
	\end{align}
	Thus, when $r \geqslant 1$, we get that
	\begin{equation}
		c_{\varphi, \varrho} + \frac{m}{2 \pi} \ln \left( r \right) \geqslant C (\varphi, \varrho),
	\end{equation}
	which completes the proof.
\end{proof}

\bigskip

\section{Reduction of order and the necessity of $m = 8 \pi$}
\label{sec:reduction}

\begin{theorem}
	\label{theorem:reduced_KS_static}
	Let $\varrho \in \cC_\KS (m, \varphi)$ be a solution of the static Keller--Segel \cref{eq:KS_static}. Furthermore assume the following bound: there exists a positive number $C$, such that on $\rl^2 - B_C (0)$, we have
	\begin{equation}
		\varrho \leqslant C r^{Cr^2}. \label[cond]{cond:hypothesis}
	\end{equation}
	Then the reduced, static Keller--Segel \cref{eq:reduced_KS_static} holds, that is $\rd \left( \ln \left( \varrho \right) - c_{\varphi, \varrho} \right) = 0$.
\end{theorem}

\smallskip

\begin{remark}
	If $\varrho \in L^\infty \left( \rl^2 \right)$, then \cref{cond:hypothesis} is trivially satisfied with $C = \max \left( 1, \| \varrho \|_{L^\infty \left( \rl^2 \right)} \right)$. We conjecture that \cref{cond:hypothesis} is not necessary in general for the conclusion \Cref{theorem:reduced_KS_static} to hold.
\end{remark}

\smallskip

\begin{remark}
	A corollary of the reduced, static Keller--Segel \cref{eq:reduced_KS_static} is that the (nonreduced) static Keller--Segel \cref{eq:KS_static} is no longer nonlocal, as $c_{\varrho, \varphi}$ can be eliminated using $\rd c_{\varphi, \varrho} = \rd \left(  \ln \left( \varrho \right) \right) = \tfrac{\rd \varrho}{\varrho}$, and get
	\begin{equation}
		\rd^* \left( \varrho \rd c_{\varphi, \varrho} \right) = - g_\varphi \left( \rd \varrho, \rd c_{\varphi, \varrho} \right) + \varrho \Delta_\varphi c_{\varphi, \varrho} = - g_\varphi \left( \rd \varrho, \tfrac{\rd \varrho}{\varrho} \right) + \varrho^2 = - \frac{|\rd \varrho|_\varphi^2}{\varrho} + \varrho^2.
	\end{equation}
	Thus the static Keller--Segel \cref{eq:KS_static} becomes
	\begin{equation}
		\Delta_\varphi \varrho + \frac{|\rd \varrho|_\varphi^2}{\varrho} - \varrho^2 = 0.
	\end{equation}
\end{remark}

\smallskip

\begin{proof}[Proof of \Cref{theorem:reduced_KS_static}]
	Let $f \eqdef \ln \left( \varrho \right) - c_{\varphi, \varrho}$. The static Keller--Segel \cref{eq:KS_static} implies that
	\begin{equation}
		\forall R \in \rl_+ : \forall \phi \in L_{1, 0}^2 \left( B_R (0), g_\varphi \right) : \quad \int\limits_{\rl^2} \varrho g_0 \left( \rd \phi, \rd f \right) \dA_0 = 0. \label{eq:new_KS_static}
	\end{equation}
	We now apply an Agmon-trick type argument: Let $\chi$ be a smooth and compactly supported function. Then, using $\phi = f \chi^2$ in the second row, we get
	\begin{align}
		\int\limits_{\rl^2} \varrho \left| \rd \left( \chi f \right) \right|^2 \dA_0	&= \int\limits_{\rl^2} \varrho \left| \rd \chi \right|^2 f^2 \dA_0 + 2 \int\limits_{\rl^2} \varrho f \chi g_0 \left( \rd \chi, \rd f \right) \dA_0 + \int\limits_{\rl^2} \varrho \chi^2 \left| \rd f \right|^2 \dA_0 \\
		&= \int\limits_{\rl^2} \varrho \left| \rd \chi \right|^2 f^2 \dA_0 + \int\limits_{\rl^2} \varrho g_0 \left( \rd \left( f \chi^2 \right), \varrho \rd f \right) \dA_0 - \int\limits_{\rl^2} \varrho \chi^2 \left| \rd f \right|^2 \dA_0 + \int\limits_{\rl^2} \varrho \chi^2 \left| \rd f \right|^2 \dA_0 \\
		&= \int\limits_{\rl^2} \varrho \left| \rd \chi \right|^2 f^2 \dA_0. \label[ineq]{ineq:Agmon}
	\end{align}
	Now for each $R \gg 1$, let $\chi = \chi_R$ be a smooth cut-off function that is 1 on $B_R (0)$, vanishes on $\rl^2 - B_{2 R} (0)$, and (for some $K \in \rl_+$) $|\rd \chi_R| = \tfrac{K}{R}$. Let $A_R = B_{2 R} (0) - B_R (0)$. Then we get that
	\begin{equation}
		\int\limits_{\rl^2} \varrho |\rd f|^2 \dA_0 \leqslant \liminf\limits_{R \rightarrow \infty} \int\limits_{\rl^2} \varrho \left| \rd \left( \chi_R f \right) \right|^2 \dA_0 = \liminf\limits_{R \rightarrow \infty} \int\limits_{\rl^2} \varrho \left| \rd \chi_R \right|^2 f^2 \dA_0 \leqslant \liminf\limits_{R \rightarrow \infty} \frac{K^2}{R^2} \int\limits_{A_R} \varrho f^2 \dA_0.
	\end{equation}
	To complete the proof, we show now that the last limit inferior is zero. Since
	\begin{equation}
		\int\limits_{A_R} \varrho f^2 \dA_0 \leqslant \left( \sqrt{\int\limits_{A_R} \varrho \ln \left( \varrho \right)^2 \dA_0} + \sqrt{\int\limits_{A_R} \varrho c_{\varphi, \varrho}^2 \dA_0} \right)^2,
	\end{equation}
	it is enough to show that both terms under the square roots are $o \left( R^2 \right)$, at least for some divergent sequence of radii. This is immediate for the second term by \Cref{lemma:c_bound}. To bound the first term, let $C$ be the constant from \cref{cond:hypothesis} and break up $A_R$ into 2 pieces:
	\begin{align}
		A_{R, I}	&\eqdef \left\{ \ x \in A_R \ \middle| \ \varrho (x) \leqslant r (x)^{- C r(x)^2} \ \right\}, \\
		A_{R, II}	&\eqdef \left\{ \ x \in A_R \ \middle| \ r (x)^{- C r(x)^2} \leqslant \varrho (x) \leqslant r (x)^{C r(x)^2} \ \right\}.
	\end{align}
	By \cref{cond:hypothesis}, $A_R = A_{R, I} \cup A_{R, II}$. Let is first inspect
	\begin{equation}
		0 \leqslant \int\limits_{A_{R, I}} \varrho \ln \left( \varrho \right)^2 \dA_0 \leqslant C (2 R)^{- C (2 R)^2} \ln \left( C (2 R)^{- C (2 R)^2} \right)^2 \Ar \left( A_{R, I}, g_0 \right) = o \left( R^2 \right).
	\end{equation}
	Finally, note that on $A_{R, II}$, we have $\left| \ln \left( \varrho \right) \right| = O \left( R^2 \ln (R) \right)$. Thus, for $R \gg 1$, we have
	\begin{equation}
		0 \leqslant \int\limits_{A_{R, II}} \varrho \ln \left( \varrho \right)^2 \dA_0	\leqslant \| \ln \left( \varrho \right) \|_{L^\infty \left( A_{R, II} \right)} \int\limits_{A_{R, II}} \varrho \left| \ln \left( \varrho \right) \right| \dA_0 \leqslant 8 C R^2 \ln \left( R \right) \int\limits_{A_{R, II}} \varrho \left| \ln \left( \varrho \right) \right| \dA_0.
	\end{equation}
	Now let $R_k \eqdef 2^k$, and then
	\begin{equation}
		0 \leqslant \frac{1}{R_k^2} \int\limits_{A_{R_k, II}} \varrho \ln \left( \varrho \right)^2 \dA_0 \leqslant 8 C \ln (2) k \int\limits_{A_{R_k, II}} \varrho \left| \ln \left( \varrho \right) \right| \dA_0.
	\end{equation}
	Since $\varrho \ln \left( \varrho \right) \in L^1 \left( \rl^2, g_0 \right)$ we have that
	\begin{equation}
		\liminf\limits_{k \rightarrow \infty} \: \left( k \int\limits_{A_{R_k, II}} \varrho \left| \ln \left( \varrho \right) \right| \dA_0 \right) = 0,
	\end{equation}
	and thus
	\begin{equation}
		0 \leqslant \int\limits_{\rl^2} \varrho \left| \rd f \right|^2 \dA_0 \leqslant \liminf\limits_{k \rightarrow \infty} \frac{K^2}{R_k^2} \int\limits_{A_{R_k}} \varrho f^2 \dA_0 = 0,
	\end{equation}
	and hence
	\begin{equation}
		\int\limits_{\rl^2} \varrho \left| \rd f \right|^2 \dA_0 = 0,
	\end{equation}
	which implies \cref{eq:reduced_KS_static}, and thus completes the proof.
\end{proof}

\smallskip

\begin{corollary}
\label{cor:rho_bounds}
	If $\varrho \in \cC_\KS$ is a solution of the static Keller--Segel \cref{eq:KS_static} and satisfies \cref{cond:hypothesis}, then there is a number $K = K (\varphi, \varrho) \geqslant 1$, such that
	\begin{equation}
		K \geqslant \varrho \left( 1 + r^2 \right)^{\frac{m}{4 \pi}} \geqslant K^{- 1}. \label[ineq]{ineq:rho_bounds}
	\end{equation}
	In particular, $\varrho \sim r^{- \frac{m}{2 \pi}}$ and $m > 4 \pi$. 
\end{corollary}

\begin{proof}
	We have
	\begin{equation}
		\ln \left( \varrho \left( 1 + r^2 \right)^{\frac{m}{4 \pi}} \right) = \ln \left( \varrho \right) + \tfrac{m}{4 \pi} \ln \left( 1 + r^2 \right) = \underbrace{\ln \left( \varrho \right) - c_{\varphi, \varrho}}_{\mbox{constant by \Cref{theorem:reduced_KS_static}}} + \underbrace{c_{\varphi, \varrho} + \tfrac{m}{4 \pi} \ln \left( 1 + r^2 \right)}_{\mbox{bounded by \Cref{lemma:c_bound}}},
	\end{equation}
	which concludes the proof.
\end{proof}

\smallskip

\begin{remark}
	\Cref{theorem:reduced_KS_static} remains true (with the same proof) even when $g_\varphi$ is replaced by any compactly supported, smooth perturbation of $g_0$. However proving \Cref{lemma:c_bound} becomes more complicated in that case, although conjecturally, that claim should still hold, and thus so should \Cref{cor:rho_bounds}.
\end{remark}

\smallskip

\begin{remark}
	Before stating our next theorem, let us recall a few facts, commonly used in literature of the Keller--Segel equations.

	First of all, and to the best of our knowledge, the only known solutions in the flat case are the ones given in \cref{eq:flat_solution}. Note that they all have mass $8 \pi$.

	A complementary fact, supporting the conjecture that static solutions must have mass $8 \pi$, is the the following "Virial Theorem" that applies to the time-dependent equation as well: Assume that $\varrho$ is a solution to the (time-dependent) Keller--Segel \cref{eq:KS1,eq:KS2}, such that for all $t$ in the domain of $\varrho$ the following quantity is finite
	\begin{equation}
		W (t) \eqdef \int\limits_{\rl^2} |x|^2 \varrho (t, x) \dA_0 (x).
	\end{equation}
	Then $W$ satisfies the following equation (cf. \cite{BDP06}*{Lemma~22} for the proof):
	\begin{equation}
		\dot{W} (t) = 4 m - \frac{m}{2 \pi}.
	\end{equation}
	In particular, if $\varrho$ is a (positive) solution to the static Keller--Segel \cref{eq:KS_static} with finite $W$, then $m = 8 \pi$. Note that for each $\varrho_{\lambda, x_\star}$ in \cref{eq:flat_solution}, we get $W = \infty$, so the above two results are indeed complementary.
\end{remark}

In the next theorem we prove that, under \cref{cond:hypothesis}, all (positive) solutions to the static Keller--Segel \cref{eq:KS_static} must have mass $8 \pi$.

\begin{theorem}
\label{theorem:critical_mass}
	If $\varrho \in \cC_\KS$ is a solution of the static Keller--Segel \cref{eq:KS_static} and satisfies \cref{cond:hypothesis}, then its mass is necessarily $8 \pi$.
\end{theorem}

\begin{proof}
	By \Cref{cor:rho_bounds}, we have that $m > 4 \pi$ and thus, for some $\epsilon > 0$, we have $\varrho = O \left( r^{- 2 - \epsilon} \right)$.

	Let now $v = (v_1, v_2)$ be a smooth, compactly supported vector field. Let us pair both sides of \cref{eq:reduced_KS_static} with $- \varrho v$, integrate over $\rl^2$ with respect to $\dA_0$ and then integrate by parts in the first term to get
	\begin{equation}
		\sum\limits_{i = 1}^2 \left( \ \int\limits_{\rl^2} \varrho \left( \del_i v_i + v_i \del_i c_{\varphi, \varrho} \right) \dA_0 \ \right) = 0. \label{eq:virial1}
	\end{equation}
	For any smooth, real function $f$, let
	\begin{equation}
		v^f (x) \eqdef \left( 2 x_1 e^{2 \varphi (x)} + \del_1 f, 2 x_2 e^{2 \varphi (x)} + \del_2 f \right),
	\end{equation}
	and let $\chi_R$ as in the proof of \Cref{theorem:reduced_KS_static}. Let us assume that $\left| \rd f \right| \in L^2 \left( \rl^2, g_\varphi \right)$. Then for $v = \chi_R v^f$ \cref{eq:virial1} becomes
	\begin{equation}
	\begin{aligned}
		0	&= \sum\limits_{i = 1}^2 \left( \ \int\limits_{\rl^2} \varrho \left( \chi_R \del_i v_i^f + \chi_R \varrho v_i^f \del_i c_{\varphi, \varrho} + \del_i \chi_R v_i^f \right) \dA_0 \ \right) \\
			&= \sum\limits_{i = 1}^2 \left( \ \int\limits_{\rl^2} \chi_R (x) \varrho (x) \left( 2 e^{2 \varphi (x)} + 4 x_i \del_i \varphi (x) e^{2 \varphi (x)} + \del_i^2 f (x) + \left( 2 x_i e^{2 \varphi (x)} + \del_i f (x) \right) \del_i c_{\varphi, \varrho} (x) \right) \dA_0 (x) \ \right) \\
			& \quad + O \left( \ \int\limits_{B_{2R} (0) - B_R (0)} |\rd \chi_R| \left| v^f \right| \varrho \dA_\varphi \right) \\
			&= 4 \underbrace{\int\limits_{\rl^2} \chi_R \varrho \dA_\varphi}_{\cI_1 (R)} + 2 \underbrace{\sum\limits_{i = 1}^2 \ \int\limits_{\rl^2} \chi_R (x) \varrho(x) x_i \del_i c_{\varphi, \varrho} (x) \dA_\varphi (x)}_{\cI_2 (R)} + \underbrace{\sum\limits_{i = 1}^2 \left( \ \int\limits_{\rl^2} \chi_R \varrho \left( 4 r \del_r \varphi - \Delta_\varphi f - g_\varphi \left( \rd f, \rd c_{\varphi, \varrho} \right) \right) \dA_0 \ \right)}_{\cI_3 (R)} \\
			& \quad + O \left( R^{- 1} \left( R + \| \rd f \|_{L^2 \left( \rl^2, g_\varphi \right)} \right) R^{- 2 - \epsilon} R^2 \right). \label{eq:virial2}
	\end{aligned}
	\end{equation}
	As $R \rightarrow \infty$ the last term goes to zero, by definition, $\cI_1 (R) \rightarrow m$. Using \cref{eq:c_def}, we get
	\begin{align}
		\cI_2 (R)	&= \sum\limits_{i = 1}^2 \ \int\limits_{\rl^2} \chi_R (x) x_i \del_i c_{\varphi, \varrho} (x) \dA_\varphi (x) \\
					&= - \frac{1}{2 \pi} \sum\limits_{i = 1}^2 \ \iint\limits_{\rl^2 \times \rl^2} \chi_R (x) \varrho (x) x_i \del_i \ln \left( |x - y| \right) \varrho (y) \dA_\varphi (y) \dA_\varphi (x) \\
					&= - \frac{1}{2 \pi} \sum\limits_{i = 1}^2 \ \iint\limits_{\rl^2 \times \rl^2} \chi_R (x) \varrho (x) x_i \frac{x_i - y_i}{|x - y|^2} \varrho (y) \dA_\varphi (y) \dA_\varphi (x),
	\end{align}
	thus
	\begin{align}
		\lim\limits_{R \rightarrow \infty} \cI_2 (R)	&= - \frac{1}{2 \pi} \sum\limits_{i = 1}^2 \ \iint\limits_{\rl^2 \times \rl^2} \varrho (x) x_i \frac{x_i - y_i}{|x - y|^2} \varrho (y) \dA_\varphi (y) \dA_\varphi (x) \\
														&= - \frac{1}{2 \pi} \sum\limits_{i = 1}^2 \ \iint\limits_{\rl^2 \times \rl^2} \varrho (x) \left( \frac{x_i - y_i}{2} + \frac{x_i + y_i}{2} \right) \frac{x_i - y_i}{|x - y|^2} \varrho (y) \dA_\varphi (y) \dA_\varphi (x) \\
														&= - \frac{1}{4 \pi} \iint\limits_{\rl^2 \times \rl^2} \varrho (x) \varrho (y) \dA_\varphi (y) \dA_\varphi (x) + \underbrace{0}_{\mbox{due to antisymmetry}} \\
														&= - \frac{m^2}{4 \pi}.
	\end{align}
	Finally, if we can choose a smooth $f$ so that
	\begin{equation}
		\Delta_\varphi f + g_\varphi \left( \rd f, \rd c_{\varphi, \varrho} \right) = 4 r \del_r \varphi, \label{eq:aux_PDE}
	\end{equation}
	and $|\rd f| \in L^2 \left( \rl^2, g_\varphi \right)$, then $\cI_3 (R) = 0$, for all $R$. For any smooth, compactly supported function $\phi$, let
	\begin{equation}
		\| \phi \|_{\varphi, \varrho} \eqdef \sqrt{ \| \rd \phi \|_{L^2 \left( \rl^2, g_\varphi \right)}^2 + \tfrac{1}{2} \| \sqrt{\varrho} \phi \|_{L^2 \left( \rl^2, g_\varphi \right)}^2},
	\end{equation}
	and let $\left( \cH_{\varphi, \varrho}, \langle - | - \rangle_{\varphi, \varrho} \right)$ the corresponding Hilbert space. Clearly $\cH_{\varphi, \varrho} \subseteq L_{1, \loc}^2$. The weak formulation of \cref{eq:aux_PDE} on $\cH_{\varphi, \varrho}$ is
	\begin{equation}
		\forall \phi \in C_{\cpt}^\infty \left( \rl^2 \right) : \quad \underbrace{\langle \rd \phi | \rd f \rangle_{L^2 \left( \rl^2, g_\varphi \right)} + \int\limits_{\rl^2} \phi g_\varphi \left( \rd f, \rd c_{\varphi, \varrho} \right) \dA_\varphi}_{B \left( f, \phi \right)} = \underbrace{\int\limits_{\rl^2} \phi r \del_r \varphi \dA_\varphi}_{\Phi_\varphi \left( \phi \right)}.
	\end{equation}
	Now if $f = \phi \in C_{\cpt}^\infty \left( \rl^2 \right)$, then
	\begin{align}
		 B \left( \phi, \phi \right)	&= \langle \rd \phi | \rd \phi \rangle_{L^2 \left( \rl^2, g_\varphi \right)} + \int\limits_{\rl^2} \phi g_\varphi \left( \rd \phi, \rd c_{\varphi, \varrho} \right) \dA_\varphi \\
		 								&= \| \rd \phi \|_{L^2 \left( \rl^2, g_\varphi \right)}^2 + \frac{1}{2} \int\limits_{\rl^2} g_\varphi \left( \rd \phi^2, \rd c_{\varphi, \varrho} \right) \dA_\varphi \\
		 								&= \| \rd \phi \|_{L^2 \left( \rl^2, g_\varphi \right)}^2 + \frac{1}{2} \int\limits_{\rl^2} \phi^2 \Delta_\varphi c_{\varphi, \varrho} \dA_\varphi \\
		 								&= \| \rd \phi \|_{L^2 \left( \rl^2, g_\varphi \right)}^2 + \frac{1}{2} \int\limits_{\rl^2} \phi^2 \varrho \dA_\varphi \\
		 								&= \| \phi \|_{\varphi, \varrho}^2,
	\end{align}
	and, using that $\varphi$ has compact support and \cref{ineq:rho_bounds}, we have
	\begin{align}
		\left| \Phi_\varphi \left( \phi \right) \right|	&= \int\limits_{\rl^2} \phi r \del_r \varphi \dA_\varphi \\
														&= \int\limits_{\rl^2} \left( \phi \sqrt{\varrho} \right) \left( \frac{r \del_r \varphi}{\sqrt{\varrho}} \right) \dA_\varphi \\
														&\leqslant \| \phi \sqrt{\varrho} \|_{L^2 \left( \rl^2, g_\varphi \right)} \sqrt{\int\limits_{\rl^2} \tfrac{r^2 \left( \del_r \varphi \right)^2}{\varrho} \dA_\varphi} \\
														&\leqslant K (\varphi, \varrho) \| \phi \|_{\varphi, \varrho}.
	\end{align}
	Thus the conditions of the Lax--Milgram theorem are satisfied and hence there is a unique $f \in \cH_{\varphi, \varrho}$ that solves \cref{eq:aux_PDE}. By elliptic regularity, $f$ is in fact smooth and by the definition $\cH_{\varphi, \varrho}$, $|\rd f| \in L^2 \left( \rl^2, g_\varphi \right)$. Hence \cref{eq:virial2} becomes $0 = 4 m - \tfrac{m^2}{2 \pi}$, which concludes the proof.
\end{proof}

\bigskip

\section{Connection to the critical Kazdan--Warner equation on the round sphere}
\label{sec:KW}

Let us assume that $\varrho \in \cC_\KS$ is a solution of the static Keller--Segel \cref{eq:KS_static} and satisfies \cref{cond:hypothesis}, and thus $m = 8 \pi$. Fix $\lambda \in \rl_+$ and $x_\star \in \rl^2$, and let $\varrho_{\lambda, x_\star}$ as in \cref{eq:flat_solution}. Pick the unique stereographic projection $p_{\lambda, x_\star} : \S^2 - \{ \mbox{ North pole } \} \rightarrow \rl^2$, so that $g_{\S^2} \eqdef \left( p_{\lambda, x_\star} \right)^* \left( \tfrac{1}{2} \varrho_{\lambda, x_\star} g_0 \right)$ is the round metric of unit radius. By \Cref{cor:rho_bounds}, the function $\widetilde{u} \eqdef \tfrac{1}{2} \ln \left( \tfrac{\varrho}{\varrho_{\lambda, x_\star}} \right)$ is bounded on $\rl^2$. Let $u \eqdef \widetilde{u} \circ p_{\lambda, x_\star} \in L^\infty \left( \S^2 \right)$. Then (omitting obvious pullbacks and computations) we have
\begin{align}
	\Delta_{\S^2} u	&= \frac{1}{\tfrac{1}{2} \varrho_{\lambda, x_\star}} \Delta_0 \left( \frac{1}{2} \ln \left( \frac{\varrho}{\varrho_{\lambda, x_\star}} \right) \right) \\
					&= \frac{1}{\varrho_{\lambda, x_\star}} \Delta_0 \left( \left( \ln \left( \varrho \right) - c_{\varphi, \varrho} \right) + c_{\varphi, \varrho} + \ln \left( \varrho_{\lambda, x_\star} \right) \right) \\
					&= \frac{1}{\varrho_{\lambda, x_\star}} \left( 0 + e^{2 \varphi} \varrho - \varrho_{\lambda, x_\star} \right) \\
					&= e^{2 \varphi} e^{2 u} - 1.
\end{align}
Since $\varphi$ is compactly supported, the pullback of $e^{2 \varphi}$ to $\S^2$ via $p_{\lambda, x_\star}$ extends smoothly over the North pole. Let us denote this extension by $h$. Then the equation on $u$ becomes
\begin{equation}
	\Delta_{\S^2} u = h e^{2 u} - 1. \label{eq:KW}
\end{equation}
This is the equation of Kazdan and Warner, \cite{kazdan_curvature_1974}*{Equation~(1.3)}, with $k = 1$ (note that they use the opposite sign convention for the Laplacian). When $\varphi$ vanishes identically, then $u = 0$ is a solution, which corresponds to the well-known $\varrho = \varrho_{\lambda, x_\star}$ solution on the flat plane. More generally, given any $\lambda \in \rl_+$ and $x_\star \in \rl^2$ and any positive scalar curvature metric $g$ on $\S^2$, one can construct a solutions to curved, static Keller-Segel \cref{eq:KS_static} as follows: by the uniformization theorem, $g$ and $g_{\S^2}$ are always conformally equivalent. Thus we have a function, $u$, that solves \cref{eq:KW} with $h$ being the scalar curvature of $g$ (pulled back under a diffeomorphism). Let now $\widetilde{u}$ and $\widetilde{h}$ be the pushforwards of $u$ and $h$, respectively, to $\rl^2$ via $p_{\lambda, x_\star}$, and let $\varrho \eqdef \varrho_{\lambda, x_\star} e^{2 \widetilde{u}}$. Then $\varrho$ solves the curved, static Keller-Segel \cref{eq:KS_static} with $\varphi = \tfrac{1}{2} \ln \left( \widetilde{h} \right)$.

\smallskip

\begin{remark}
	Using the reduced, static Keller--Segel \cref{eq:reduced_KS_static} also, equations similar to the Kazdan--Warner \cref{eq:KW} were studied in \cites{BCN17,WWY19}. These equations however are still on the plane so the geometric interpretation above is lost.
\end{remark}

\smallskip

Unfortunately, \cref{eq:KW} is the critical version of the Kazdan--Warner equation in \cite{kazdan_curvature_1974}. Thus we cannot, in general, assume solvability for an arbitrary $h$. In fact, Kazdan and Warner found a necessary conditions for the existence of solutions: For each spherical harmonic of degree one, $u_1$, by \cite{kazdan_curvature_1974}*{Equation~(8.10)}, we have
\begin{equation}
	\int\limits_{\S^2} g_{\S^2} \left( \rd u_1, \rd h \right) e^{2 u} \omega_{\S^2} = 0, \label{eq:KS_obstruction}
\end{equation}
where $\omega_{\S^2}$ is the symplectic/area form of $g_{\S^2}$. We use \cref{eq:KS_obstruction} to prove the following:

\begin{theorem}
	\label{theorem:nonexistence}
	There exists $\varphi \in C_\cpt^\infty \left( \rl^2 \right)$, arbitrarily close to the identically zero function, such that the static Keller--Segel \cref{eq:KS_static} has no solutions satisfying \cref{cond:hypothesis}.
\end{theorem}

\begin{proof}
	Let us assume that $\varphi$ is radial (with respect to $x_\star$). Then $h$ is only a function of the polar angle $\theta \in \left( - \tfrac{\pi}{2}, \tfrac{\pi}{2} \right)$, on $\S^2$. When $u_1 = \sin \left( \theta \right)$, then \cref{eq:KS_obstruction} becomes
	\begin{equation}
		\int\limits_{\S^2} \cos \left( \theta \right) \left( \del_\theta h \right) e^{2 u} \omega_{\S^2} = 0. \label{eq:KW_obstraction_radial}
	\end{equation}
	Since $\del_\theta h \sim e^{2 \varphi} \del_r \varphi$, we get that if $\varphi$ is nonconstant and $\del_r \varphi$ is either nonnegative or nonpositive, then \cref{eq:KW_obstraction_radial} cannot hold. This concludes the proof.
\end{proof}

\bigskip

\section{The variation aspects of the Keller--Segel theory on curved planes}

We end this paper with a complementary result to \Cref{theorem:critical_mass}, showing that the energy functional (formally) corresponding to the Keller--Segel flow in \cref{eq:KS1,eq:KS2} is bounded from below only when $m = 8 \pi$. In order to do that, we first prove a curved version of the logarithmic Hardy--Littlewood--Sobolev inequality.

\subsection{Curved logarithmic Hardy--Littlewood--Sobolev inequality and the Keller--Segel free energy}
\label{sec:curved_log-HLS}

Let $\lambda \in \rl_+$ and $x_\star \in \rl^2$, and define
\begin{equation}
	\mu_{\lambda, x_\star} (x) \eqdef \frac{\lambda^2}{\pi \left( \lambda^2 + |x - x_\star|^2 \right)^2}. \label{eq:mu_def}
\end{equation}
Then $\mu_{\lambda, x_\star}$ is everywhere positive, $\int_{\rl^2} \mu_{\lambda, x_\star} \dA_0 = 1$, and for any $f \in C_{\cpt}^\infty \left( \rl^2 \right)$
\begin{equation}
	\lim\limits_{\lambda \rightarrow 0} \ \int\limits_{\rl^2} \mu_{\lambda, x_\star} f \dA_0 = f (x_\star). \label{eq:dirac}
\end{equation}
The following identities about $\mu_{\lambda, x_\star}$ are easy to verify:
\begin{subequations}
\begin{align}
	\int\limits_{\rl^2} m \mu_{\lambda, x_\star} \ln \left( m \mu_{\lambda, x_\star} \right) \dA_0							&= m \ln \left( \frac{m}{\pi e} \right) - 2 m \ln (\lambda), \label{eq:mu_entropy} \\
	\int\limits_{\rl^2} G (\cdot, y) \mu_{\lambda, x_\star} (y) \dA_0 (y)													&= \frac{1}{8 \pi} \left( \ln \left( \mu_{\lambda, x_\star} \right) - 2 \ln (\lambda) + \ln (\pi) \right), \label{eq:G_mu} \\
	\iint\limits_{\rl^2 \times \rl^2} \mu_{\lambda, x_\star} (x) G (x, y) \mu_{\lambda, x_\star} (y) \dA_0 (x) \dA_0 (y)	&= - \frac{1}{2 \pi} \ln (\lambda) - \frac{1}{4 \pi}. \label{eq:mu_Coulomb_energy}
\end{align}
\end{subequations}
Now we can state the \emph{logarithmic Hardy--Littlewood--Sobolev inequality} on $\left( \rl^2, g_0 \right)$, which is a special case of \cite{beckner_sharp_1993}*{Theorem~2}.

\begin{theorem}
	\label{theorem:classical_log-HLS}
	Let $\varrho$ be an almost everywhere positive function on $\rl^2$ and assume that
	\begin{equation}
		\int\limits_{\rl^2} \varrho \dA_0 = m \in \rl_+.
	\end{equation}
	Then for all $\lambda \in \rl_+$, $x_\star \in \rl^2$, we have
	\begin{equation}
		\int\limits_{\rl^2} \varrho (x) \ln \left( \frac{\varrho (x)}{m \mu_{\lambda, x_\star} (x)} \right) \dA_0 \geqslant \frac{4 \pi}{m} \iint\limits_{\rl^2 \times \rl^2} (\varrho (x) - m \mu_{\lambda, x_\star} (x)) G (x, y) (\varrho (y) - m \mu_{\lambda, x_\star} (y)) \dA_0 (x) \dA_0 (y). \label[ineq]{ineq:classical_log-HLS}
	\end{equation}
	Moreover, equality holds exactly when $\varrho = m \mu_{\lambda, x_\star}$.
\end{theorem}

\begin{proof}[Idea of the proof:]
	Note that \cref{eq:mu_entropy,eq:mu_Coulomb_energy,eq:mu_entropy} imply that \cref{ineq:classical_log-HLS} is equivalent to
	\begin{equation}
		\int\limits_{\rl^2} \varrho (x) \ln \left( \varrho (x) \right) \dA_0 + \frac{2}{m} \iint\limits_{\rl^2 \times \rl^2} \varrho (x) \ln \left( |x - y| \right) \varrho (y) \dA_0 (x) \dA_0 (y) + m \left( 1 + \ln (\pi) - \ln (m) \right) \geqslant 0. \label[ineq]{ineq:classical_log-HLS_alternative}
	\end{equation}
	Now \cref{ineq:classical_log-HLS_alternative} is the $n = 2$ and $f = g$ case of \cite{beckner_sharp_1993}*{inequality (27)}.
\end{proof}

\smallskip

Let now $g$ be \emph{any} smooth Riemannian metric on $\rl^2$, not necessarily conformally equivalent to $g_0$. There still exists a smooth function, $\varphi$, such that if the area form of $g$ is $\dA_g$, then
\begin{equation}
	\dA_g = e^{2 \varphi} \dA_0. \label{eq:area_form_change}
\end{equation}
For the remainder of this section (but this section only), let $\varphi$ be defined via \cref{eq:area_form_change}, and write, as before $\dA_\varphi \eqdef \dA_g$. When $g$ is not conformally equivalent to $g_0$, then $G$ is no longer the Green's function for $g$. Now let $\mu_{\lambda, x_\star}^\varphi \eqdef \mu_{\lambda, x_\star} e^{- 2 \varphi}$. Note that $\int_{\rl^2} \mu_{\lambda, x_\star}^\varphi \dA_\varphi = 1$.

The next lemma is a generalization of \Cref{theorem:classical_log-HLS}.

\begin{lemma}
	\label{lem:curved_log-HLS}
	Let $\varrho$ be an almost everywhere positive function on $\rl^2$ and assume that
	\begin{equation}
		\int\limits_{\rl^2} \varrho \dA_\varphi = m \in \rl_+.
	\end{equation}
	Then for all $\lambda \in \rl_+$ and $x_\star \in \rl^2$, we have
	\begin{equation}
		\int\limits_{\rl^2} \varrho \ln \left( \frac{\varrho}{m \mu_{\lambda, x_\star}^\varphi} \right) \dA_\varphi \geqslant \frac{4 \pi}{m} \iint\limits_{\rl^2 \times \rl^2} \left( \varrho (x) - m \mu_{\lambda, x_\star}^\varphi (x) \right) G (x, y) \left( \varrho (y) - m \mu_{\lambda, x_\star}^\varphi (y) \right) \dA_\varphi (x) \dA_\varphi (y), \label[ineq]{ineq:curved_log-HLS}
	\end{equation}
	and equality holds exactly when $\varrho = m \mu_{\lambda, x_\star}^\varphi$.
\end{lemma}

\begin{proof}
	Let us first rewrite the left-hand side of \cref{ineq:curved_log-HLS}:
	\begin{equation}
		\int\limits_{\rl^2} \varrho \ln \left( \frac{\varrho}{m \mu_{\lambda, x_\star}^\varphi} \right) \dA_\varphi = \int\limits_{\rl^2} \varrho \ln \left( \frac{\varrho}{m \mu e^{- 2 \varphi}} \right) e^{2 \varphi} \dA_0 = \int\limits_{\rl^2} \left( \varrho e^{2 \varphi} \right) \ln \left( \frac{\left( \varrho e^{2 \varphi} \right)}{m \mu} \right) \dA_0. \label[ineq]{ineq:first_ineq}
	\end{equation}
	Since $\varrho e^{2 \varphi}$ is almost everywhere positive and
	\begin{equation}
		\int\limits_{\rl^2} \left( \varrho e^{2 \varphi} \right) \dA_0 = \int\limits_{\rl^2} \varrho \dA_\varphi = m,
	\end{equation}
	we can use \cref{ineq:classical_log-HLS}, with $\varrho$ replaced by $\varrho e^{2 \varphi}$, and get
	\begin{equation}
		\int\limits_{\rl^2} \left( \varrho e^{2 \varphi} \right) \ln \left( \frac{\left( \varrho e^{2 \varphi} \right)}{m \mu} \right) \dA_0 \geqslant \frac{4 \pi}{m} \iint\limits_{\rl^2 \times \rl^2} \left( \varrho (x) e^{2 \varphi (x)} - m \mu (x) \right) G (x, y) \left( \varrho (y) e^{2 \varphi (y)} - m \mu (y) \right) \dA_0 (x) \dA_0 (y). \label[ineq]{ineq:intermediate_ineq}
	\end{equation}
	Furthermore
	\begin{multline}
		\iint\limits_{\rl^2 \times \rl^2} \left( \varrho (x) e^{2 \varphi (x)} - m \mu (x) \right) G (x, y) \left( \varrho (y) e^{2 \varphi (y)} - m \mu (y) \right) \dA_0 (x) \dA_0 (y) \\
		= \iint\limits_{\rl^2 \times \rl^2} \left( \varrho (x) - m \mu (x) e^{- 2 \varphi (x)} \right) G (x, y) \left( \varrho (y) - m \mu (y) e^{- 2 \varphi (y)} \right) \left( e^{2 \varphi (x)} \dA_0 (x) \right) \left( e^{2 \varphi (y)} \dA_0 (y) \right) \\
		= \iint\limits_{\rl^2 \times \rl^2} \left( \varrho (x) - m \mu_{\lambda, x_\star}^\varphi (x) \right) G (x, y) \left( \varrho (y) - m \mu_{\lambda, x_\star}^\varphi (y) \right) \dA_\varphi (x) \dA_\varphi (y). \label{eq:last_eq}
	\end{multline}
	Combining \cref{ineq:first_ineq,ineq:intermediate_ineq,eq:last_eq} proves \cref{ineq:curved_log-HLS}. Finally, equality in \cref{ineq:intermediate_ineq} holds exactly when $\varrho e^{2 \varphi} = m \mu$, or equivalently, when $\varrho = m \mu_{\lambda, x_\star}^\varphi$, which conclude the proof.
\end{proof}

\smallskip

\begin{remark}
	As opposed to the flat case, when $\varphi$ is not identically zero, the $m = 8 \pi$ minimizer for the curved logarithmic Hardy--Littlewood--Sobolev \cref{ineq:curved_log-HLS}, $8 \pi \mu_{\lambda, x_\star}^\varphi$, is \emph{not} a solution to the static Keller--Segel \cref{eq:KS_static}, nor the reduced, static Keller--Segel \cref{eq:reduced_KS_static}. Instead, we get
	\begin{equation}
		\rd \left( \ln \left( 8 \pi \mu_{\lambda, x_\star}^\varphi \right) - c_{\varphi, 8 \pi \mu_{\lambda, x_\star}^\varphi} \right) = \rd \left( \ln \left( 8 \pi \mu_{\lambda, x_\star} \right) - 2 \varphi - c_{0, 8 \pi \mu_{\lambda, x_\star}} \right) = - 2 \rd \varphi \not \equiv 0.
	\end{equation}
\end{remark}

\bigskip

\subsection{The Keller--Segel free energy}
\label{sec:KS_fe}

The (flat) \emph{Keller--Segel free energy} of $\varrho \in \cC_\KS (m, 0)$ is
\begin{equation}
	\cF_0 \left( \varrho \right) = \int\limits_{\rl^2} \varrho \ln \left( \varrho \right) \dA_0 - \frac{1}{2} \iint\limits_{\rl^2 \times \rl^2} \varrho (x) G (x, y) \varrho (y) \dA_0 (x) \dA_0 (y). \label{eq:KS_free_energy}
\end{equation}

\smallskip

\begin{remark}
	Formally, \cref{eq:KS1} is the negative gradient flow of the Keller--Segel free energy under the \emph{Wasserstein metric}. Formally this metric can be introduced as follows: If $\varrho \in \cC_\KS (m, \varphi)$, then the operator $f \mapsto L_\varrho (f) \eqdef \rd^* \left( \varrho \rd f \right)$ is expected to be nondegenerate. Then if $\dot{\varrho}$ is a tangent vector to $\cC_\KS (m, \varphi)$, then its Wasserstein norm is given by
	\begin{equation}
		\| \dot{\varrho} \|_W^2 \eqdef \int\limits_{\rl^2} \dot{\varrho} L_\varrho^{- 1} \left( \dot{\varrho} \right) \dA_0.
	\end{equation}
	Then the Wasserstein norm is a Hilbert norm, thus can be used to define gradient flows.
\end{remark}

\begin{remark}
	The functional in \eqref{eq:KS_free_energy} is also the energy of self-gravitating Brownian dust; cf. \cite{CRRS04}.
\end{remark}

\smallskip

Let us generalize $\cF_0$ to $\left( \rl^2, g_\varphi \right)$: For any $\varrho \in \cC_\KS (m, \varphi)$, let the \emph{curved Keller--Segel free energy} be
\begin{equation}
	\cF_\varphi \left( \varrho \right) \eqdef \int\limits_{\rl^2} \varrho \ln \left( \varrho \right) \dA_\varphi - \frac{1}{2} \iint\limits_{\rl^2 \times \rl^2} \varrho (x) G (x, y) \varrho (y) \dA_\varphi (x) \dA_\varphi (y). \label{eq:curved_KS_free_energy}
\end{equation}

Now we are ready to prove our last main result.

\begin{theorem}
	\label{theorem:KS_bound}
	The curved Keller--Segel free energy \eqref{eq:curved_KS_free_energy} is bounded from below on $\cC_\KS (m, \varphi)$, exactly when $m = 8 \pi$.
\end{theorem}

\begin{proof}
	Let $m, \lambda \in \rl_+$, and $\mu_{\lambda, 0}$ as in \cref{eq:mu_def} (with $x_\star = 0$). Now \cref{eq:mu_Coulomb_energy,eq:mu_entropy} imply that
	\begin{align}
		\cF_\varphi \left( m \mu_{\lambda, x_\star} e^{- 2 \varphi} \right)	&= \int\limits_{\rl^2} m \mu_{\lambda, x_\star} e^{- 2 \varphi} \ln \left( m \mu_{\lambda, x_\star} e^{- 2 \varphi} \right) \dA_\varphi \\
		& \quad - \frac{1}{2} \iint\limits_{\rl^2 \times \rl^2} m \mu_{\lambda, x_\star} (x) e^{- 2 \varphi (x)} G (x, y) m \mu_{\lambda, x_\star} (y) e^{- 2 \varphi (y)} \dA_\varphi (x) \dA_\varphi (y) \\
		&= \int\limits_{\rl^2} m \mu_{\lambda, x_\star} \ln \left( m \mu_{\lambda, x_\star} \right) \dA_0 - 2 m \int\limits_{\rl^2} \mu_{\lambda, x_\star} \varphi \dA_0 \\
		& \quad - \frac{m^2}{2} \iint\limits_{\rl^2 \times \rl^2} \mu_{\lambda, x_\star} (x) G (x, y) \mu_{\lambda, x_\star} (y) \dA_0 (x) \dA_0 (y) \\
		&= \frac{m}{4 \pi} \left( m - 8 \pi \right) \ln (\lambda) + m \ln \left( \frac{m}{\pi e} \right) - 2 m \int\limits_{\rl^2} \mu_{\lambda, x_\star} \varphi \dA_0.
	\end{align}
	As $\lambda \rightarrow 0^+$, the last term goes to $\varphi (x_\star)$. Thus, when $m > 8 \pi$, then
	\begin{equation}
		\lim\limits_{\lambda \rightarrow 0^+} \cF_\varphi \left( m \mu_{\lambda, x_\star} e^{- 2 \varphi} \right) = - \infty.
	\end{equation}
	Similarly, as $\lambda \rightarrow \infty$, the last term goes to zero. Thus, when $m < 8 \pi$, then
	\begin{equation}
		\lim\limits_{\lambda \rightarrow \infty} \cF_\varphi \left( m \mu_{\lambda, x_\star} e^{- 2 \varphi} \right) = - \infty.
	\end{equation}
	This proves the claim for $m \neq 8 \pi$.

	When $m = 8 \pi$, then for any $\varrho \in \cC_\KS (m, \varphi)$, we have
	\begin{align}
		\cF_\varphi \left( \varrho \right) &= \int\limits_{\rl^2} \varrho \ln \left( \varrho \right) \dA_\varphi - \frac{1}{2} \iint\limits_{\rl^2 \times \rl^2} \varrho (x) G (x, y) \varrho (y) \dA_\varphi (x) \dA_\varphi (y) \\
		&= \int\limits_{\rl^2} \varrho \ln \left( \varrho \right) \dA_\varphi - \frac{4 \pi}{m} \iint\limits_{\rl^2 \times \rl^2} \varrho (x) G (x, y) \varrho (y) \dA_\varphi (x) \dA_\varphi (y) \\
		&= \int\limits_{\rl^2} \varrho \ln \left( \frac{\varrho}{m \mu_{\lambda, x_\star}^\varphi} \right) \dA_\varphi - \frac{4 \pi}{m} \iint\limits_{\rl^2 \times \rl^2} \left( \varrho (x) - m \mu_{\lambda, x_\star}^\varphi (x) \right) G (x, y) \left( \varrho (y) - m \mu_{\lambda, x_\star}^\varphi (y) \right) \dA_\varphi (x) \dA_\varphi (y) \\
		& \quad + m \ln (m) - 2 \int\limits_{\rl^2} \varrho \varphi \dA_\varphi + \int\limits_{\rl^2} \varrho \ln \left( \mu_{\lambda, x_\star} \right) \dA_\varphi - 8 \pi \iint\limits_{\rl^2 \times \rl^2} \varrho (x) G (x, y) \mu_{\lambda, x_\star} (y) \dA_\varphi (x) \dA_0 (y) \\
		& \quad + 4 \pi m^2 \iint\limits_{\rl^2 \times \rl^2} \mu_{\lambda, x_\star} (x) G (x, y) \mu_{\lambda, x_\star} (y) \dA_0 (x) \dA_0 (y) \\
		&= \int\limits_{\rl^2} \varrho \ln \left( \frac\varrho{m \mu_{\lambda, x_\star}^\varphi} \right) \dA_\varphi - \frac{4 \pi}{m} \iint\limits_{\rl^2 \times \rl^2} \left( \varrho (x) - m \mu_{\lambda, x_\star}^\varphi (x) \right) G (x, y) \left( \varrho (y) - m \mu_{\lambda, x_\star}^\varphi (y) \right) \dA_\varphi (x) \dA_\varphi (y) \\
		& \quad + m \ln (m) - 2 \int\limits_{\rl^2} \varrho \varphi \dA_\varphi + \int\limits_{\rl^2} \varrho (x) \left( \ln \left( \mu_{\lambda, x_\star} (x) \right) - 8 \pi \int\limits_{\rl^2} G (x, y) \mu_{\lambda, x_\star} (y) \dA_0 (y) \right) \dA_\varphi (x) \\
		& \quad + 4 \pi m \iint\limits_{\rl^2 \times \rl^2} \mu_{\lambda, x_\star} (x) G (x, y) \mu_{\lambda, x_\star} (y) \dA_0 (x) \dA_0 (y).
	\end{align}
	Now, using \cref{ineq:curved_log-HLS,eq:G_mu,eq:mu_Coulomb_energy,eq:dirac}, and plugging back $m = 8 \pi$, we get
	\begin{equation}
		\inf \: \left( \left\{ \ \cF_\varphi \left( \varrho \right) \ \middle| \ \varrho \in \cC_\KS (m, \varphi) \ \right\} \right) = 8 \pi \ln \left( \tfrac{8}{e} \right) - 16 \pi \sup \left( \left\{ \varphi (x) \middle| x \in \rl^2 \right\} \right),
	\end{equation}
	which completes the proof.
\end{proof}

\begin{remark}
	It is not entirely obvious if the relevant generalization of Keller--Segel free energy \eqref{eq:KS_free_energy} is the functional, $\cF_\varphi$, in \cref{eq:curved_KS_free_energy}. There is an generalization that is minimally coupled to the metric: Let $\kappa_\varphi \eqdef \Delta_\varphi \varphi$ be the Gauss curvature of $g_\varphi$ and $q \in \rl$ be a coupling constant. Then let us define
	\begin{align}
		\cF_{\varphi, q} \left( \varrho \right)	&\eqdef \int\limits_{\rl^2} \varrho \ln \left( \varrho \right) \dA_\varphi - \frac{1}{2} \iint\limits_{\rl^2 \times \rl^2} \varrho (x) G (x, y) \varrho (y) \dA_\varphi (x) \dA_\varphi (y) \\
												& \quad + q \iint\limits_{\rl^2 \times \rl^2} \kappa_\varphi (x) G (x, y) \varrho (y) \dA_\varphi (x) \dA_\varphi (y). \label{eq:coupled_KS_free_energy}
	\end{align}
	When $m \neq 8 \pi$, the proof of \Cref{theorem:KS_bound} can still be used the prove the unboundedness of $\cF_{\varphi, q}$, and when $m = 8 \pi$, we get
	\begin{align}
		\cF_{\varphi, q} \left( \varrho \right)	&\geqslant \int\limits_{\rl^2} \varrho \ln \left( \frac\varrho{m \mu_{\lambda, x_\star}^\varphi} \right) \dA_\varphi - \frac{4 \pi}{m} \iint\limits_{\rl^2 \times \rl^2} \left( \varrho (x) - m \mu_{\lambda, x_\star}^\varphi (x) \right) G (x, y) \left( \varrho (y) - m \mu_{\lambda, x_\star}^\varphi (y) \right) \dA_\varphi (x) \dA_\varphi (y) \\
		& \quad + (q - 2) \int\limits_{\rl^2} \varrho \varphi \dA_\varphi \dA_\varphi + 8 \pi \ln \left( \tfrac{8}{e} \right).
	\end{align}
	In particular, when $q = 2$, then $\varrho = 8 \pi \mu_{\lambda, x_\star}^\varphi$ is an absolute minimizer of $\cF_{\varphi, q}$.
\end{remark}

	%========================
	\bibliography{references}
	%========================

\end{document}